\documentclass[10pt]{article}
\usepackage{amsmath, amsthm, amssymb, amsfonts, mathrsfs, mathtools,calc}
\usepackage{fdsymbol}
\usepackage{graphicx}
\usepackage{makeidx}
\usepackage[left=2cm,right=2cm,top=2cm,bottom=2cm]{geometry}
\usepackage{caption}
\usepackage{subcaption}
\usepackage{tikz}
\usetikzlibrary{decorations.pathreplacing}
\tikzstyle{vertex}=[auto=left,circle,draw=black,fill=white, inner sep=1.5]

\newtheorem{theorem}{Theorem}[section]

\newtheorem{prop}[theorem]{Proposition}

\newtheorem{lemm}{Lemma}[section]

\newtheorem{ex}{Example}[section]

\title{Signed Chromatic Polynomials of Signed Book Graphs}

\author{ Deepak\\
Department of Mathematics\\
Indian Institute of Technology Guwahati\\
Guwahati, India - 781039\\
Email: deepakmath55555@iitg.ac.in\\
\\
Bikash Bhattacharjya\\
Department of Mathematics\\
Indian Institute of Technology Guwahati\\
Guwahati, India - 781039\\
Email: b.bikash@iitg.ac.in
}

\begin{document}
\maketitle

\vspace{-0.3cm}
\noindent
\textbf{Abstract}.  In 2015, Matthias Beck and his team developed a computer program in SAGE which efficiently determines the number of signed proper $k$-colorings for a given signed graph. In this article, we determine the number of different signatures on Book graph up to switching isomorphisms. We also find a recursive formula of the signed chromatic polynomials of signed Book graphs.

\noindent {\textbf{Keywords}: Signed graph, balance, switching isomorphism, Book graph, signed chromatic number, signed chromatic polynomial. 

\section{Introduction}\label{intro}

Signed graph is a graph with positive or negative sign label to its edges. Signed graph has been evolved as a generalisation of ordinary graph because all edges of ordinary graph can be considered positive or negative. In a signed graph, set of negative edges is called the signature of that signed graph. Further, it is clear that if a graph has $n$ vertices and $m$ edges then there are $2^{m}$ ways to put a sign on its edges. Some mathematicians have determined the exact number of signatures on some graphs up to switching isomorphism. Zaslavsky~\cite{T.Zaslavsky} had shown that there are only six non-isomorphic signatures on Petersen graph. V.Sivaraman~\cite{Sivaraman} determined that there are only seven signed Heawood graphs up to switching isomorphism. Up to switching isomorphism, there are two signed $K_3$'s, three signed $K_4$'s and seven signed $K_5$'s. We infer that finding the exact number of non-isomorphic signatures on some graphs help us to understand the graph more rigorously. Zaslavsky~\cite{Zaslavsky2} introduced the notion of signed graph coloring of signed graphs and he emphasized that for coloring a signed graph, signed colors are needed. Further, there is a chromatic polynomial of signed graphs with similar enumerative structure as for ordinary graphs. In~\cite{Beck2}, Zaslavsky also showed that signed graphs have two signed chromatic polynomials. In \cite{Beck}, the authors published a SAGE code which produces the signed chromatic polynomial as output when a signed graph is given as input. In \cite{Zaslavsky3} and \cite{Brian}, authors showed that how the hyperplane arrangements are useful to find the signed chromatic polynomial of a given signed graph. We will use this hyperplane arrangements idea to calculate the signed chromatic polynomials of some signed graphs, and these chromatic polynomials will be used to determine the signed chromatic polynomials of signed book graph. In this paper, we determine the number of different signatures on Book graph up to switching isomorphism, and also calculate the signed chromatic polynomials of signed Book graphs. We divide this paper into two parts. The first part determines the signatures on Book graph up to switching isomorphisms. The second part determines the signed chromatic polynomials of signed Book graphs.

\section{Preliminaries}\label{prelim}

A \textit{signified graph} is a graph $G$ together with an assignment of $+$ or $-$ signs to its edges. If $\Sigma$ is the set of negative edges, then we denote the signified graph by $(G, \Sigma)$. The set $\Sigma$ is called the signature of $(G, \Sigma)$. Signature $\Sigma$ can also  be viewed as a function from $E(G)$ into $\{+1, -1\}$. A \textit{resigning (switching)} of a signified graph at a vertex $v$ is to change the sign of each edge incident to $v$. We say $(G,\Sigma_{2})$ is \textit{switching equivalent}, or simply equivalent, to $(G,\Sigma_{1})$ if it is obtained from $(G,\Sigma_{1})$ by a sequence of resignings. In other words, we say that $(G,\Sigma_{2})$ is (switching) equivalent to $(G,\Sigma_{1})$ if there exists a function $f : V \rightarrow \{+1, -1\}$ such that $\Sigma_{2}(e) = f(u)\Sigma_{1}(e)f(v)$ for each edge $e=uv$. Switching defines an equivalence relation on the set of all signified graphs over $G$ (also on the set of signatures). Each equivalence class of this equivalence relation is called a \textit{signed graph} and is denoted by $[G,\Sigma]$, where $(G,\Sigma)$ is any member of the class. Many properties of signified graphs are well known to be invariant under switching.

\begin{prop} \cite{Naserasr} 
  If $G$ has $m$ edges, $n$ vertices and $c$ components, then there are $2^{(m-n+c)}$ distinct signed graphs over $G$.
\end{prop}

We say that a cycle in a signed graph is \emph{balanced} if the product of signs of edges of that cycle is positive and unbalanced, otherwise. One of the first theorems in the theory of signed graphs is that the set of unbalanced cycles uniquely determines the class of signed graphs to which a signified graph belongs. More precisely, we state the following theorem.

\begin{theorem}\cite{Zaslavsky}
\label{Signature}
Two signatures $\Sigma_{1}$ and $\Sigma_{2}$ are switching equivalent if and only if they have the same set of unbalanced cycles.
\end{theorem}

Throughout this paper, all unsigned graphs are simple. However, parallel edges or negative loops will be allowed for signed graphs. A solid line in a signed graph represents positive edge and a dashed line represents negative edge. For all other graph theoretic terms that are used but not defined in this article, we refer the reader to~\cite{Bondy}.

\section{Definitions and Notations}\label{notations}

For $m \geq 3$ and $n \geq 1$, the $m$-cycle \textit{Book graph} $B(m,n)$ consists of $n$ copies of the cycle $C_m$ \linebreak[4] with one common edge. The copies of the cycle $C_m$ are called the pages of $B(m,n)$. Let \linebreak[4] $V(B(m,n)) = \{u,~v\} \cup \{u_{k}^l~|~1 \leq l \leq n,~~ 1 \leq k \leq m-2 \}$, so that $uv$ is the edge common to the cycles $C_{m}^{l}$, where $C_{m}^{l} = uu_{1}^{l}u_{2}^{l}u_{3}^{l}...u_{m-3}^{l}u_{m-2}^{l}vu$, for $1 \leq l \leq n$. For example, the cycle $C_{4}^{1}$ in $B(4,3)$ is the cycle $uu_{1}^{1}u_{2}^{1}vu$, where the graph $B(4,3)$ is shown in Figure \ref{BG1}.

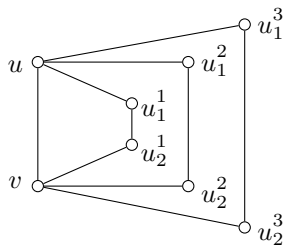
\begin{figure}[ht]
\centering
\begin{tikzpicture}[scale=0.5]
\node[vertex] (v1) at (9,2) {};
\node [below] at (8.4,2.5) {$v$};
\node[vertex] (v2) at (9,5.3) {};
\node [below] at (8.4,5.6) {$u$};
\node[vertex] (v3) at (11.5,3.1) {};
\node [below] at (12.1,3.6) {$u_{2}^{1}$};
\node[vertex] (v4) at (11.5,4.2) {};
\node [below] at (12.1,4.8) {$u_{1}^{1}$};
\node[vertex] (v5) at (13,2) {};
\node [below] at (13.7,2.5) {$u_{2}^{2}$};
\node[vertex] (v6) at (13,5.3) {};
\node [below] at (13.7,6) {$u_{1}^{2}$};
\node[vertex] (v7) at (14.5,0.9) {};
\node [below] at (15.2,1.5) {$u_{2}^{3}$};
\node[vertex] (v8) at (14.5,6.3) {};
\node [below] at (15.2,7) {$u_{1}^{3}$};

\foreach \from/\to in {v1/v2,v1/v3,v2/v4,v1/v5,v2/v6,v1/v7,v2/v8,v3/v4,v5/v6,v7/v8} \draw (\from) -- (\to);

\end{tikzpicture}
\caption{The Book graph $B(4,3)$.}
\label{BG1}
\end{figure}

Let $\text{Aut}(G)$ denotes the automorphism group of a graph $G$. It is clear that $B(m,1)=C_m$, and therefore $\text{Aut}(B(m,1))\cong D_m$, the dihedral group of order $2m$. Also, $B(m,n) \setminus \{u, v\}$ is just a disjoint union of $n$ copies of the path $P_{m-2}$ on $m-2$ vertices. It is clear that any permutation of these $n$ copies of $P_{m-2}$ determines an automorphism of $B(m,n)$ in the obvious way. Conversely, any non-trivial automorphism of $B(m,n)$ permutes the vertices $u$ and $v$, and also permutes these $n$ copies of $P_{m-2}$. Hence it is easy to see that $\text{Aut}(B(m,n)) \cong S_n \times S_2$, for $n \geq 2$. Thus every automorphism of $B(m,n)$ can only permute the vertices $u$ and $v$, and permute the $n$ pages of the graph. Note that the interchanges of $u$ and $v$ induces interchanges of pair of vertices of each page of the graph.

We say that two signatures $\Sigma_1$ and $\Sigma_2$ on a graph $G$ are \textit{automorphic} if there exists an automorphism $f$ of $G$ such that $uv \in \Sigma_1$ if and only if $f(u)f(v) \in \Sigma_2$. If two signatures are automorphic then they are said to be \textit{automorphic type} signatures. Thus the signatures $\Sigma_1$ and $\Sigma_2$ of a graph $G$ are distinct automorphic type signatures if $\Sigma_1$ is not automorphic to $\Sigma_2$. For instance, the signatures $\{uu^{2}_1\}$ and $\{uu^{1}_1\}$ are automorphic type signatures of $B(4,3)$. However, the signatures $\{uv\}$ and $\{uu^{1}_1\}$ are distinct automorphic type signatures of $B(4,3)$.

We say that two signified graphs $(G, \Sigma_{1})$ and $(H, \Sigma_{2})$ are \textit{isomorphic} if there exists a graph isomorphism $\psi : V(G) \rightarrow V(H)$ which preserve the edge signs. We denote it by $\Sigma_{1} \cong \Sigma_{2}$. Further, they are said to be \textit{switching isomorphic} if $\Sigma_{1}$ is isomorphic to a switching of $\Sigma_{2}$. That is, there exists a representation $(H, \Sigma_{2}^{\prime})$ which is equivalent to $(H, \Sigma_{2})$ such that $\Sigma_{1}$ $\cong$ $\Sigma_{2}^{\prime}$. We denote it by $\Sigma_{1} \sim \Sigma_{2}$. 

\section{Signings on Book Graph}\label{book}

The automorphism group of $B(m,n)$ is discussed in Section~\ref{notations}. We will use this automorphism group to determine the number of different signatures on Book graphs up to switching isomorphisms. Note that a cycle can be either balanced or unbalanced. Therefore, a cycle can have only two signatures up to switching, \textit{viz.,} an empty signature or a signature of size one. If a cycle is unbalanced, then we can make any pre-chosen edge of the cycle negative by suitable switchings. On the basis of this observation, we have the following proposition.

\begin{prop}
\label{prop1}
A signature on $B(m,n)$ is either an empty signature or all the edges of the signature are incident to `$u$' up to switchings. Moreover, the size of such a signature on $B(m,n)$ is at most $\lceil {n/2} \rceil$.
\end{prop}
\begin{proof}
Let $\Sigma$ be a signature on $B(m,n)$. If $\Sigma = \emptyset$, then we are done. Now, let $|\Sigma| \geq 1$. Recall that the graph $B(m,n)$ is the union of $n$ copies of $m$-cycles and their intersection is an edge. For each $l=1,2, \cdots, n $, if the cycle $C_{m}^{l}$ in $\left(B(m,n),\Sigma \right)$ is balanced, then we can make all its edges positive by suitable switchings. If the cycle $C_{m}^{l}$ is unbalanced, then we can make the edge $uu_{1}^{l}$ or $uv$ negative and rest of the edges positive by switchings. This proves the first part of the proposition. 

Note that $d(u) = n+1$ in the graph $B(m,n)$. We already found that a signature on $B(m,n)$ is either empty signature or each of its edges is incident to the vertex $u$, up to switchings. Thus, by switching at $u$, if needed, we find a signature of $B(m,n)$ of size  at most $\lceil {n/2} \rceil$. This completes the proof of the proposition. 
\end{proof}

Recall that two signed graphs $[G, \Sigma_{1}]$ and $[H, \Sigma_{2}]$ are isomorphic if there exists a graph isomorphism $\phi : V(G) \rightarrow V(H)$ which preserve the edge signs. Thus if two signed graphs have different number of unbalanced cycles of same length, then they cannot be isomorphic to each other. In the following theorem, we use this fact to compute the number of non-isomorphic signatures on $B(m,n)$. 

Let us denote an empty signature by $\Sigma_0$. Denote by $\Sigma_{l}$ a signature on $B(m,n)$ of size $l$ that does not contain the edge $uv$ but all of its edges are incident to the vertex $u$. Similarly, $\Sigma_{l}^{uv}$ denotes a signature on $B(m,n)$ of size $l$ containing the edge $uv$ and remaining $l-1$ edges are also incident to the vertex $u$.

\begin{theorem}\label{theorem1}
Up to switching isomorphisms, the number of distinct signatures on $B(m,n)$ is $n+1$.
\end{theorem}
\begin{proof}
We consider two cases according as $n$ is odd or even.\\
\noindent
\textbf{Case 1($n=2k+1$)}. According to Proposition~\ref{prop1}, the size of a signature on $B(m,2k+1)$ is at most $k+1$, and each edge of such a signature is incident to vertex $u$, up to switchings. It is easy to see that any two signatures of size $l$ which do not contain the edge $uv$ are automorphic, and $\Sigma_{l} = \{uu_{1}^{1}, uu_{1}^{2}, \cdots , uu_{1}^{l}\}$ is one of such signatures, where $1 \leq l \leq k$. Further, any two signatures of size $l$ which contain the edge $uv$ are automorphic, and $\Sigma_{l}^{uv} = \{uv, uu_{1}^{1}, \cdots , uu_{1}^{l-1}\}$ is one of such signatures, where $1 \leq l \leq k$. It is clear that the number of unbalanced $m$-cycles in $\left(B(m,2k+1),\Sigma_{l}\right)$ and $\left(B(m,2k+1),\Sigma_{l}^{uv}\right)$ are $l$ and $(2k+1)-(l-1) = 2k+2-l$, respectively. Since $1\leq l \leq k$, the numbers $l$ and $2k+2-l$ cannot be same. This shows that the signatures $\Sigma_{l}$ and $\Sigma_{l}^{uv}$ have different number of unbalanced $m$-cycles. Hence for each $l$, the signatures $\Sigma_{l}$ and $\Sigma_{l}^{uv}$ cannot be isomorphic, where $1 \leq l \leq k$.

 Further, any two signatures of size $k+1$ which do not contain the edge $uv$ are automorphic to each other and any two signatures of size $k+1$ which contain the edge $uv$ are also automorphic to each other. Since $d(u) = 2k+2$ in $B(m,2k+1)$, resigning at $u$ transforms $\Sigma_{k+1}$ to a signature automorphic to $\Sigma_{k+1}^{uv}$. Hence we have only one signature on $B(m,2k+1)$ of size $k+1$ up to switching isomorphism. Further, the number of unbalanced $m$-cycles in $[B(m,2k+1), \Sigma_{k+1}]$ is $k+1$. So the signature $\Sigma_{k+1}$ is not isomorphic to $\Sigma_{l}$ or $\Sigma_{l}^{uv}$, where $1 \leq l \leq k$.
 
Thus the signatures $\Sigma_{0}, \Sigma_{l}, \Sigma_{l}^{uv}$ and $\Sigma_{k+1}$ are pairwise non-switching isomorphic, where $1 \leq l \leq k$. This shows that $B(m,2k+1)$ has exactly $2k+2$ non-switching isomorphic signatures.

\noindent
\textbf{Case 2 ($n=2k$)}.
Proposition~\ref{prop1} tells us that the size of a signature on $B(m,2k)$ is at most $k$, and each edge of such a signature is incident to vertex $u$, up to switchings. It is easy to see that any two signatures of size $l$ which do not contain the edge $uv$ are automorphic and the number of unbalanced $m$-cycles in $B(m,2k)$ with such a signature is $l$. Similarly, any two signatures of size $l$ which contain the edge $uv$ are automorphic and the number of unbalanced $m$-cycles in $B(m,2k)$ with such a signature is $2k-(l-1) = 2k+1-l$, where $1 \leq l \leq k$. The numbers $l$ and $2k+1-l$ can never be same, since $l$ satisfies $1 \leq l \leq k$. Thus for each $l$, the signatures $\Sigma_{l}$ and $\Sigma_{l}^{uv}$ on $B(m,2k)$ are non-switching isomorphic, where $1 \leq l \leq k$. This proves that the signatures $\Sigma_{0}, \Sigma_{l}$ and $\Sigma_{l}^{uv}$ are pairwise non-switching isomorphic, where $1 \leq l \leq k$. Hence the graph $B(m,2k)$ has exactly $2k+1$ non-switching isomorphic signatures.\\
This completes the proof of the theorem.
\end{proof}

\section{Preliminaries for Signed Coloring}\label{signed-coloring}
Recall that a \textit{signed graph} $(G,\sigma)$ consists of an unsigned graph $G$, whose vertex set is $V(G)$ and edge set is $E(G)$, and a sign function $\sigma$ which labels each edge as positive or negative. A coloring of an ordinary graph in $k$ colors is a mapping of the vertex set of the graph into the set $[k] = \{1, 2,...,k\}$. However, for the coloring of signed graphs we must have signed colors (see~\cite{Zaslavsky2} for details). If $(G,\sigma)$ is a signed graph, we define a signed coloring of $(G,\sigma)$ in $2k+1$ \textit{signed colors} to be a mapping 
$$c~:~V(G) \longrightarrow \{-k, -k+1,...,0,...,k-1,k\}.$$ 

A signed coloring is \textit{zero-free} or \emph{balanced} if it never takes the value zero. To know more about the difference between these two signed colorings \textit{viz.,} coloring and zero-free coloring, see~\cite{Zaslavsky2}. A signed coloring of a signed graph is \textit{proper} if $c(v) \neq \sigma(e)c(u)$ whenever there is an edge $e = uv$. The condition $c(v) \neq \sigma(e)c(u)$ implies that $c(v) \neq 0$ if there is a negative loop at $v$.

The \textit{chromatic polynomial} $\chi_{(G,\sigma)}(\lambda)$ of a signed graph $(G,\sigma)$ is the function defined for odd positive arguments $\lambda = 2k+1$, whose value equals the number of proper signed colorings of $(G,\sigma)$ in $2k+1$ signed colors. Similarly, the \textit{balanced chromatic polynomial} $\chi_{(G,\sigma)}^{b}(\lambda)$ of $(G,\sigma)$, defined for even positive arguments $\lambda = 2k$, is the function whose value equals the number of proper zero-free signed colorings of $(G,\sigma)$ in $2k$ signed colors.

In~\cite{Zaslavsky2}, the authors proved that the chromatic number and the chromatic polynomials of a signed graph are invariant under switchings. In perfect analogy to ordinary graph coloring theory, we have the following theorems.

\begin{theorem}~\cite{Zaslavsky2}
If $(G,\sigma)$ is a signed graph on $n$ vertices, then $\chi_{(G,\sigma)}(\lambda)$ and $\chi_{(G,\sigma)}^b(\lambda)$ are monic polynomial functions of $\lambda$ of degree $n$.
\end{theorem}

Let $e$ be a positive edge in $(G,\sigma)$. The edge-contraction $(G,\sigma)/e$ is obtained by identifying the end points of $e$ and deleting $e$. We also have the signed analogue of edge deletion-contraction formula of chromatic polynomial for simple unsigned graph.  

\begin{theorem}~\cite{Zaslavsky2} \label{E-C}
Let $(G,\sigma)$ be a signed graph and $e$ be a positive edge in $(G,\sigma)$. Then  
$$\chi_{(G,\sigma)}(\lambda) = \chi_{(G,\sigma) \setminus e}(\lambda) - \chi_{(G,\sigma) /e}(\lambda)~\text{ and }~\chi_{(G,\sigma)}^b(\lambda) = \chi_{(G,\sigma) \setminus e}^b(\lambda) - \chi_{(G,\sigma) /e}^b(\lambda).$$
\end{theorem} 

\subsection{Hyperplane Arrangements}\label{hyperplane}
Coloring of graphs and signed graphs have a geometrical interpretation via hyperplane arrangements. Recall that a coloring $c$ of a simple graph $G$, where $V(G) = \{1, 2, \cdots ,n\}$, is proper if $c(i) \neq c(j)$, whenever there is an edge $ij$ in $G$. We denote $c(i)$ by $c_i$ and consider $c$ as the point $(c_1, c_2, \cdots, c_n)$ in the real affine space $\mathbb{R}^{n}$, and call it proper if it does not lie in any of the hyperplanes $h_{ij} : x_{i} = x_{j}$ for each $ij \in E(G)$. That is, if we write $\mathcal{H}(G) = \{h_{ij} : ij \in E(G)\}$, which is the hyperplane arrangement of the graph $G$, then counting proper colorings of $G$ using $k$ colors is same as counting integral points in $[k]^{n} \setminus \underset{ij \in E(G)}{{\bigcup}} h_{ij}$ in the space $\mathbb{R}^n$.

The type-$BC$ Coxeter arrangement of dimension $n$, denoted by $BC_{n}$, consists of the following hyperplanes: $h_{ij}^{+} := \{x \in \mathbb{R}^n : x_i = x_j \}$, $h_{ij}^{-} := \{x \in \mathbb{R}^n : x_i = -x_j \}$, and $h_{i} := \{x \in \mathbb{R}^n : x_i = 0 \}$, where $1 \leq i < j \leq n$. For a signed graph $(G,\sigma)$, the signed graphic arrangement $B_{(G,\sigma)}$ is the sub-arrangement of $BC_{n}$ that encodes the properness conditions of signed coloring of $(G,\sigma)$. It is the collection $B_{(G,\sigma)} :     = \lbrace h_{ij}^{\sigma (ij)} : {ij \in E(G)} \rbrace$, where we consider $h_{ii}^{\pm}$ as $h_i$. Thus a coloring $c$ of a signed graph is proper if and only if, as a point in $\mathbb{Z}^{n}$, it avoids each hyperplane of $B_{(G,\sigma)}$.

To each sub collection of an arrangement $\mathcal{A}$, we evaluate its intersection and call this a \textit{flat} of the arrangement. Flats of any arrangement have a partial order by reverse containment. The intersection poset $L\mathcal{(A)}$ of an arrangement $\mathcal{A}$ with respect to partial order of reverse containment is called the \textit{intersection lattice}, whose elements are the flats of the arrangement. Note that $\mathbb{R}^{n}$ itself is a flat of an arrangement $\mathcal{A}$, that corresponds to the sub-collection $\emptyset$ of $\mathcal{A}$. See~\cite{Zaslavsky3} for details.

A \textit{rank} function defined on lattices of arrangements which maps each element to a non-negative integer such that $\text{rank}(\hat{0}) = 0$, and whenever $p$ immediately succeeds $q$ (i.e., there exist no elements between $p$ and $q$), $\text{rank}(p) = \text{rank}(q) +1$. Note that $\hat{0}$ denotes the smallest element of intersection lattice of the arrangement. Since $\hat{0}$ in $L\mathcal{(A)}$ is $\mathbb{R}^{n}$, we observe that a flat of rank $k$ has dimension $n-k$.

Let $\mathcal{H}(G)$ and $\mathcal{H}(K_n)$ be the hyperplane arrangements of a simple graph $G$ and the complete graph $K_n$ on the vertex set $\{1,2,\cdots,n\}$, respectively. Denote the difference $\mathcal{H}(K_n) \setminus \mathcal{H}(G)$ by $\mathcal{C}(G)$. We define the poset $\mathcal{P}(G)$ corresponding to $G$ as the intersection lattice of $\mathcal{C}(G)$. For instance, the poset $\mathcal{P}(P_3)$ of the path $P_3$ is shown as a Hasse diagram in Figure~\ref{lat1}.

For a poset $\mathcal{P}$ equipped with a rank function, the \textbf{Whitney number} $w_{i}^{\mathcal{P}}$ is defined to be the number of elements of $\mathcal{P}$ of rank $i$.
The following theorem determines the chromatic polynomial of a simple graph $G$ in terms of the Whitney numbers $w_i^{\mathcal{P}(G)}$ and the falling factorial $(k)_{i} \coloneqq k(k-1)(k-2) \cdots (k-i+1)$.

\begin{theorem}~\cite{Brian}
\label{Brian1}
For a simple graph $G$ on $n$ vertices, the chromatic polynomial is given by 
\begin{equation*}
\chi_{G}(k) = \sum_{i=0}^{n} w_{i}^{\mathcal{P}(G)} (k)_{n-i}.
\end{equation*}
\end{theorem}

We give an example to explain the usage of Theorem~\ref{Brian1}.
\begin{ex}\label{ex1}
 Let $P_3$ be the path on three vertices such that $V(P_3) = \{1,2,3\}$ and $E(P_3) = \{12,23\}$. Its hyperplane arrangement is given by $\mathcal{H}(P_3) = \{h_{12}, h_{23}\}$. Therefore, $\mathcal{C}(P_3) = \{h_{13}\}$, and hence $\mathcal{P}(P_3)$ is as shown in Figure~\ref{lat1}. It is clear that, $w_{0}^{\mathcal{P}(P_3)} = 1, w_{1}^{\mathcal{P}(P_3)} = 1, w_{2}^{\mathcal{P}(P_3)} = 0$, and $w_{3}^{\mathcal{P}(P_3)} = 0$. Thus, by Theorem~\ref{Brian1}, the chromatic polynomial of $P_3$ is given by 
\begin{equation*}
\chi_{P_3}(k)  = (k)_3 + (k)_2  = k(k-1)(k-2) + k(k-1) = k(k-1)^{2}.
\end{equation*}
In the similar way, one can show that $\chi_{P_m}(k)  =  k(k-1)^{m-1}$, where $P_m$ is the path on $m$ vertices.
\end{ex}

\begin{figure}[ht]
\centering
\begin{tikzpicture}
\node[vertex] (v1) at (0,0) {};
\node[vertex] (v1) at (0,2.01) {}; 
\draw [thick]
(0,0) -- (0,2);
\node at (0,-0.4) {$\mathbb{R}^3$};
\node at (0,2.3) {$x_1 = x_3$};
\end{tikzpicture}
\caption{The poset $\mathcal{P}(P_3)$} \label{lat1}
\end{figure}
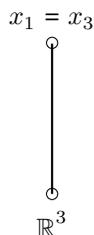

Let $(G,\sigma)$ be a signed graph on the vertex set $\{1,2,\cdots,n\}$ and let $B_{(G,\sigma)}$ be its signed graphic hyperplane arrangement. Denote the difference $BC_{n} \setminus B_{(G,\sigma)}$ by $\mathcal{C}(G,\sigma)$. The poset $\mathcal{P}(G,\sigma)$ corresponding to the signed graph $(G,\sigma)$ is defined to be the intersection lattice of $\mathcal{C}(G,\sigma)$. For instance, the poset $\mathcal{P}(C_{2}^{-})$ of an unbalanced cycle of length two is shown as a Hasse diagram in Figure~\ref{lat2}. 

\begin{figure}[ht]
\centering
\begin{tikzpicture}
\node[vertex] (v1) at (0,0) {};
\node[vertex] (v1) at (1,1) {};
\node[vertex] (v1) at (-1,1) {};
\draw [thick] (0,0) -- (1,1);
\draw [thick] (0,0) -- (-1,1);
\node at (0,-0.4) {$\mathbb{R}^2$};
\node at (1,1.3) {$h_2$};
\node at (-1,1.3) {$h_1$};
\end{tikzpicture}
\caption{The poset $\mathcal{P}(C_{2}^{-})$} \label{lat2}
\end{figure}
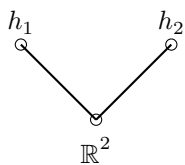

The following theorem determines the chromatic polynomial of a signed graph $(G,\sigma)$ in terms of \textbf{Whitney numbers} of the Poset $\mathcal{P}(G,\sigma)$ and the falling factorial $(k)_{i}$.

\begin{theorem}\cite{Brian}
\label{Brian2}
For a signed graph $(G,\sigma)$ on $n$ vertices, the signed chromatic polynomial is given by 
\begin{equation*}
\chi_{(G,\sigma)}(2k+1) = \sum_{i=0}^{n} w_{i}^{P(G,\sigma)} 2^{n-i} (k)_{n-i}.
\end{equation*}
\end{theorem}

\begin{ex}\label{ex2}
 Let $C_{2}^{-}$ be an unbalanced cycle of length two as shown in Figure~\ref{C_2}. Its hyperplane arrangement is $B_{C_{2}^{-}} = \{h_{12}^{+}, h_{12}^{-}\}$, and therefore $\mathcal{C}(C_{2}^{-}) = BC_{2} \setminus B_{C_{2}^{-}} = \{h_1, h_2\}$. Observe that $w_{0}^{\mathcal{P}(C_{2}^{-})} = 1$, $w_{1}^{\mathcal{P}(C_{2}^{-})} = 2$ and $w_{2}^{\mathcal{P}(C_{2}^{-})} = 0$. Using Theorem~\ref{Brian2}, the signed chromatic polynomial $\chi_{C_{2}^{-}}(2k+1)$ is obtained as 
\begin{equation*}
\chi_{C_{2}^{-}}(2k+1)  = 2^2(k)_2 + 2 \cdot 2(k)_1  = 4k(k-1) + 4k = (2k)^2.
\end{equation*}
\end{ex}

\begin{figure}[ht]
\centering
\begin{tikzpicture}
\node[vertex] (v1) at (0,0) {};
\node [below] at (0,-0.2) {1};
\node[vertex] (v2) at (3,0) {};
\node [below] at (3,-0.2) {2};
\draw [dashed] (3,0) to node [sloped,above] { } (0,0);
\draw (3,0) to[out=150,in=30] node [sloped,above] { } (0,0);
\end{tikzpicture} 
\caption{An unbalanced cycle of length two.} \label{C_2}
\end{figure}
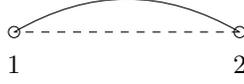

\section{Chromatic Polynomials of Signed Book Graphs}\label{chrom-poly}

Let $B_{l}(m,n)$ and $B^{*}_{l}(m,n)$ denote the signed book graphs with signature $\{uu_{1}^{1}, uu_{1}^{2}, \cdots , uu_{1}^{l}\}$ and \linebreak[4] $\{uv, uu_{1}^{1}, \cdots , uu_{1}^{l-1}\}$, respectively, where $1 \leq l \leq n $. Note that $B_{l}(m,n)$ does not contain the edge $uv$, whereas $B^{*}_{l}(m,n)$ contains the edge $uv$, and that each of their signatures is of size $l$. Proposition~\ref{prop1} along with an application of switching at the vertex $u$, if needed, gives now the following lemma.

\begin{lemm}
\begin{itemize}
\item[(i)] If $n=2k$, then the maximum size of a signature on $B(m,2k)$ is $k$ up to switchings, and $B^{*}_{k}(m,2k)$ is switching equivalent to $B_{k+1}(m,2k)$. 
\item[(ii)] If $n=2k+1$, then the maximum size of a signature on $B(m,2k+1)$ is $k+1$ up to switchings, and $B^{*}_{k+1}(m,2k+1)$ is switching equivalent to $B_{k+1}(m,2k+1)$.
\end{itemize}
\end{lemm}

Let us denote an unbalanced cycle on $n$ vertices by $C_n^-$. In Example~\ref{ex2}, we found that the chromatic polynomial of an unbalanced two cycle is $\chi_{C_{2}^{-}}(2k+1) = 4k^2$. In~\cite{Beck}, the authors calculated that \linebreak[4] $\chi_{C_{3}^{-}}(2k+1) = 8k^3$. In the following lemma, we determine the chromatic polynomial of $C_n^-$.

\begin{lemm} \label{C-}
The chromatic polynomial of an unbalanced cycle $C_{n}^{-}$ is $\chi_{C_{n}^{-}}(2k+1) = (2k)^{n}$, where $n \geq 2$. 
\end{lemm}
\begin{proof}
We prove this lemma by induction on $n$. If $n = 2$, the result is true by Example~\ref{ex2}. Let us assume that the result holds for all $n \leq m-1$, where $m \geq 3$. We shall prove that the result is also true for $n = m$. If $e$ is a positive edge of $C_{m}^{-}$ then by edge deletion-contraction formula, we get 
\begin{equation} \label{eq1}
\chi_{C_{m}^{-}}(2k+1) = \chi_{P_{m}}(2k+1) - \chi_{C_{m-1}^{-}}(2k+1).
\end{equation}
We know that $\chi_{P_m}(2k+1) = (2k+1)(2k)^{m-1}$ and by induction hypothesis, we have \linebreak[4] $\chi_{C_{m-1}^{-}}(2k+1) = (2k)^{m-1}$. So from Equation~(\ref{eq1}), we have
\begin{equation*}
 \chi_{C_{m}^{-}}(2k+1) = \chi_{P_m}(2k+1) - \chi_{C_{m-1}^{-}}(2k+1) = (2k+1)(2k)^{m-1} - (2k)^{m-1} = (2k)^{m}.
\end{equation*}
Thus, the proof follows by induction.
\end{proof}

The following lemma is a key to calculate the chromatic polynomial of a signed graph $(G,\sigma)^{l+1}$ in terms of chromatic polynomial of $(G,\sigma)$, where $(G,\sigma)^{l+1}$ is obtained from a given signed graph $(G,\sigma)$ by attaching a path $P_{l+1}=uu_1u_2\ldots u_l$ to a vertex $u$ of $G$.

\begin{lemm} \label{Path_add}
If $(G,\sigma)$ be a signed graph then $\chi_{(G,\sigma)^{l+1}}(2k+1) = (2k)^{l} \chi_{(G,\sigma)}(2k+1)$.
\end{lemm}
\begin{proof}
We prove this lemma by induction on $l$. It is known that $\chi_{P_{l+1}}(2k+1) = (2k+1)(2k)^{l}$. Let $l = 1$, and $e_1 = uu_{1}$ be attached to a vertex $u$ of $G$. Using edge deletion-contraction formula on $e_1$, we get 
\begin{equation*}
\chi_{(G,\sigma)^{2}}(2k+1) = (2k+1)  \chi_{(G,\sigma)}(2k+1) - \chi_{(G,\sigma)}(2k+1) = (2k)  \chi_{(G,\sigma)}(2k+1).
\end{equation*}
Let us assume that the result holds for $l = m-1$, that is, $\chi_{(G,\sigma)^{m}}(2k+1) = (2k)^{m-1}  \chi_{(G,\sigma)}(2k+1)$, where $m \geq 2$. Now let $l = m$ and $e = uu_{1}$. Using edge deletion-contraction formula on the edge $e$ of $(G,\sigma)^{m+1}$, we get $\chi_{(G,\sigma)^{m+1}}(2k+1) = \chi_{P_{m} \cup (G,\sigma)}(2k+1) - \chi_{(G,\sigma)^{m}}(2k+1)$, where $P_m$ and $(G,\sigma)$ are disjoint. Thus we have
\begin{align*}
\chi_{(G,\sigma)^{m+1}}(2k+1) & = (2k+1)(2k)^{m-1}  \chi_{(G,\sigma)}(2k+1) - \chi_{(G,\sigma)^{m}}(2k+1) \\
 & = (2k+1)(2k)^{m-1}  \chi_{(G,\sigma)}(2k+1) - (2k)^{m-1}  \chi_{(G,\sigma)}(2k+1) \\
 & = (2k)^{m}  \chi_{(G,\sigma)}(2k+1).
\end{align*}
Hence the proof follows by induction. 
\end{proof}
 
For each $m \geq 3$ and $n \geq 1$, consider the Book graph $B(m,n)$ and replace the edge $uv$ by an unbalanced cycle of length two. The signed graph so obtained is denoted by $B_{m}^{n}$. For example, the graphs $B_{4}^{3}$ and $B_{m}^{1}$ are shown in Figure~\ref{Figure5}.

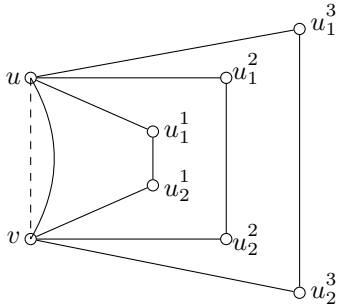
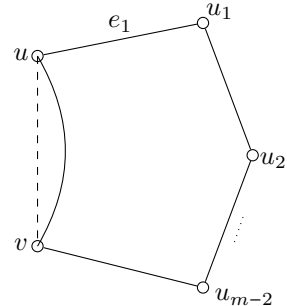
\begin{figure}[ht]
\begin{subfigure}{0.33\textwidth}
\begin{tikzpicture}[scale=0.65]
\node[vertex] (v1) at (9,2) {};
\node [below] at (8.65,2.35) {$v$};
\node[vertex] (v2) at (9,5.3) {};
\node [below] at (8.65,5.6) {$u$};
\node[vertex] (v3) at (11.5,3.1) {};
\node [below] at (12,3.6) {$u_{2}^{1}$};
\node[vertex] (v4) at (11.5,4.2) {};
\node [below] at (12,4.8) {$u_{1}^{1}$};
\node[vertex] (v5) at (13,2) {};
\node [below] at (13.4,2.5) {$u_{2}^{2}$};
\node[vertex] (v6) at (13,5.3) {};
\node [below] at (13.4,6) {$u_{1}^{2}$};
\node[vertex] (v7) at (14.5,0.9) {};
\node [below] at (15,1.5) {$u_{2}^{3}$};
\node[vertex] (v8) at (14.5,6.3) {};
\node [below] at (15,7) {$u_{1}^{3}$};

\foreach \from/\to in {v1/v3,v2/v4,v1/v5,v2/v6,v1/v7,v2/v8,v3/v4,v5/v6,v7/v8} \draw (\from) -- (\to);

\draw [dashed] (9,5.3) to node [sloped,above] { } (9,2);
\draw (9,5.3) to[out=-60,in=60] node [sloped,above] { } (9,2);

\end{tikzpicture}
\caption{The graph $B_{4}^{3}$.}\label{unbal_2}
\end{subfigure}
\hfill
\begin{subfigure}{0.33\textwidth}
\begin{tikzpicture}[scale=1.1]
\node[vertex] (v1) at (9,2) {};
\node [below] at (8.8,2.2) {$v$};
\node[vertex] (v2) at (9,4.3) {};
\node [below] at (8.8,4.5) {$u$};
\node[vertex] (v7) at (11,1.5) {};
\node [below] at (11.5,1.6) {$u_{m-2}$};
\node[vertex] (v8) at (11,4.7) {};
\node [below] at (11.2,5.05) {$u_{1}$};
\node[vertex] (v9) at (11.6,3.1) {};
\node [below] at (11.85,3.25) {$u_{2}$};
\node [below] at (10,4.9) {$e_1$};
\draw[dotted](11.38,2.05) -- (11.5,2.40);

\foreach \from/\to in {v1/v7,v2/v8,v7/v9,v8/v9} \draw (\from) -- (\to);
\draw [dashed] (9,4.3) to node [sloped,above] { } (9,2);
\draw (9,4.3) to[out=-60,in=60] node [sloped,above] { } (9,2);

\end{tikzpicture}
\caption{The graph $B_{m}^{1}$.} \label{figure6.1}
\end{subfigure}
\caption{The graphs $B_{4}^{3}$ and $B_{m}^{1}$.}\label{Figure5}
\end{figure}

As a convention, let us take $B_{2}^{1} := C_{2}^{-}$. So $\chi_{B_{2}^{1}}(2k+1) = \chi_{C_{2}^{-}}(2k+1) = (2k)^{2}$. Further, as an application of Lemma~\ref{Path_add} we have the following lemma.

\begin{lemm} \label{B_m}
For $m \geq 2$, the signed chromatic polynomial of $B_{m}^{1}$ is given by 
$$\chi_{B_{m}^{1}}(2k+1) = \Big[ \sum_{i=0}^{m-2} (-1)^{i} (2k)^{(m-2)-i} \Big]  \chi_{C_{2}^{-}}(2k+1) = \sum_{i=0}^{m-2} (-1)^{i} (2k)^{m-i}.$$
\end{lemm}
\begin{proof}
We prove this lemma by induction on $m$. Consider $B_{m}^{1}$ as given in Figure~\ref{Figure5}(b). For $m=2$, the result is true by Example~\ref{ex2}.

\begin{figure}[ht]
\begin{subfigure}{0.24\textwidth}
\begin{tikzpicture}[scale=0.7]
\node[vertex] (v1) at (9,2) {};
\node [below] at (8.7,2.3) {$v$};
\node[vertex] (v2) at (9,4.3) {};
\node [below] at (8.7,4.6) {$u$};
\node[vertex] (v7) at (11,1.5) {};
\node [below] at (11.65,1.75) {$u_{r-2}$};
\node[vertex] (v8) at (11,4.7) {};
\node [below] at (11.4,5.2) {$u_{1}$};
\node[vertex] (v9) at (11.6,3.1) {};
\node [below] at (12,3.35) {$u_{2}$};
\node [below] at (14.5,3.5) {$\longrightarrow$};
\node [below] at (10,5.1) {$e_1$};
\draw[dotted](11.5,2.0) -- (11.62,2.35);

\foreach \from/\to in {v1/v7,v2/v8,v7/v9,v8/v9} \draw (\from) -- (\to);

\draw [dashed] (9,4.3) to node [sloped,above] { } (9,2);
\draw (9,4.3) to[out=-60,in=60] node [sloped,above] { } (9,2);

\end{tikzpicture}
\end{subfigure}
\hfill
\begin{subfigure}{0.24\textwidth}
\begin{tikzpicture}[scale=0.8]
\node[vertex] (v1) at (9,2) {};
\node [below] at (8.7,2.2) {$v$};
\node[vertex] (v2) at (9,4.3) {};
\node [below] at (8.7,4.5) {$u$};
\node[vertex] (v7) at (10.3,2) {};
\node [below] at (10.3,1.95) {$u_{r-2}$};
\node[vertex] (v8) at (11.6,2) {};
\node [below] at (11.8,1.95) {$u_{2}$};
\node[vertex] (v9) at (12.9,2) {};
\node [below] at (12.9,1.95) {$u_{1}$};
\node [below] at (14,3) {$-$};
\node [below] at (11.15,1.85) {$\cdots$};

\foreach \from/\to in {v1/v7,v7/v9,v8/v9} \draw (\from) -- (\to);

\draw [dashed] (9,4.3) to node [sloped,above] { } (9,2);
\draw (9,4.3) to[out=-60,in=60] node [sloped,above] { } (9,2);

\end{tikzpicture}
\end{subfigure}
\hfill
\begin{subfigure}{0.24\textwidth}
\begin{tikzpicture}[scale=0.8]
\node[vertex] (v1) at (9,2) {};
\node [below] at (8.7,2.2) {$v$};
\node[vertex] (v2) at (9,4.3) {};
\node [below] at (8.7,4.5) {$u$};
\node[vertex] (v7) at (11,1.5) {};
\node [below] at (11.6,1.65) {$u_{r-2}$};
\node[vertex] (v8) at (11,4.7) {};
\node [below] at (11.4,5.1) {$u_{2}$};
\node[vertex] (v9) at (11.6,3.1) {};
\node [below] at (11.95,3.35) {$u_{3}$};
\node [below] at (10,5.1) {$e_2$};
\draw[dotted](11.4,2.0) -- (11.62,2.55);

\foreach \from/\to in {v1/v7,v2/v8,v7/v9,v8/v9} \draw (\from) -- (\to);

\draw [dashed] (9,4.3) to node [sloped,above] { } (9,2);
\draw (9,4.3) to[out=-60,in=60] node [sloped,above] { } (9,2);

\end{tikzpicture}
\end{subfigure}
\caption{An application of edge deletion-contraction on $B_{r}^{1}$.} \label{Un_1}
\end{figure}
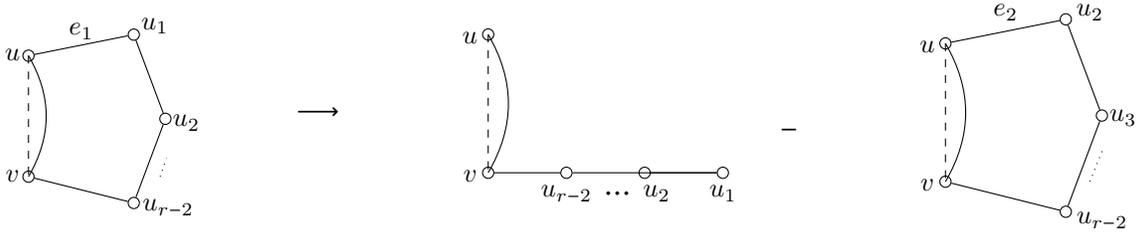

\noindent
Now let us assume that the result is true for $m = r-1$, where $r \geq 3$. That is, 
\begin{equation} \label{eq2}
\chi_{B_{r-1}^{1}}(2k+1) = \sum_{i=0}^{r-3} (-1)^{i} (2k)^{r-1-i}.
\end{equation}
An application of edge deletion-contraction on $e_1 = uu_{1}$ of $B_{r}^{1}$ is shown in Figure~\ref{Un_1}. Using Lemma~\ref{Path_add}, the chromatic polynomial of the graph in the middle of Figure~\ref{Un_1} can be computed easily, since $\chi_{C_{2}^{-}}(2k+1)$ is known. By induction hypothesis, the chromatic polynomial of the third graph of Figure~\ref{Un_1} is given in Equation~(\ref{eq2}). Therefore we find that
\begin{align*}
\chi_{B_{r}^{1}}(2k+1) & = (2k)^{r-2}  \chi_{C_{2}^{-}}(2k+1) - \sum_{i=0}^{r-3} (-1)^{i} (2k)^{r-1-i} \\
 & = (-1)^{0}(2k)^{r-2}  (2k)^{2} + \sum_{j=1}^{r-2} (-1)^{j} (2k)^{r-j} \\
 & =  \sum_{j=0}^{r-2} (-1)^{j} (2k)^{r-j} .
\end{align*}
Hence the lemma is proved by induction.
\end{proof}

The chromatic polynomial $\chi_{B_{m}^{1}}(2k+1)$, repeated use of edge deletion-contraction formula and Lemma~\ref{Path_add} give us the following proposition.

\begin{prop}\label{prop_1}
For $n \geq 2$, the chromatic polynomial of $B_{m}^{n}$ is given by 
\begin{equation*}
 \chi_{B_{m}^{n}}(2k+1) = \Big[ \sum_{i=0}^{m-2} (-1)^{i} (2k)^{(m-2)-i} \Big] \chi_{B_{m}^{n-1}}(2k+1).
\end{equation*} 
\end{prop} 
It is well known that, for given ordinary graphs $G$ and $H$, where $G \cap H$ is a complete graph, we have 
\begin{equation} \label{eq3}
\chi_{G \cup H}(k) = \frac{\chi_{G}(k) \chi_{H}(k)}{\chi_{G \cap H}(k)}.
\end{equation}

Using Equation~(\ref{eq3}), we now determine the chromatic polynomial of an unsigned Book graph.

\begin{theorem} \label{B_0}
The chromatic polynomial of $B(m,n)$, where $m \geq 3$ and $n \geq 2$, is given by 
\begin{equation*}
\chi_{B(m,n)}(k) = \frac{ [(k-1)^{m} + (-1)^{m}(k-1)]^{n}}{ [k(k-1)]^{n-1} }.
\end{equation*}
\end{theorem}
\begin{proof}
We prove this theorem by induction on $n$. It is clear that $B(m,1) = C_m$, a cycle on $m$ vertices, and it is well known that
\begin{equation} \label{eq4}
\chi_{C_m}(k) = (k-1)^{m} + (-1)^{m}(k-1).   
\end{equation}
Further, it is clear that the graph $B(m,2)$ is the union of two $m$-cycles whose intersection is $K_2$. Thus by Equation~(\ref{eq3}), we get 
\begin{equation*}
\chi_{B(m,2)}(k) = \frac{ [(k-1)^{m} + (-1)^{m}(k-1)]^{2}}{k(k-1)}.
\end{equation*} 
This shows that the result is true for $n=1$ and $n = 2$. Let us assume that the result is true for  $n = r-1$, where $r \geq 3$. That is, 
\begin{equation} \label{eq5}
\chi_{B(m,r-1)}(k) = \frac{ [(k-1)^{m} + (-1)^{m}(k-1)]^{r-1}}{ [k(k-1)]^{r-2} }.
\end{equation}
Now we prove that the result is true for $n = r$. Note that the graph $B(m,r)$ is the union of the graphs $B(m,r-1)$ and $C_m$, whose intersection is $K_2$. Therefore by Equation~(\ref{eq3}), we have 
\begin{equation} \label{eq6}
\chi_{B(m,r)}(k) = \frac{\chi_{B(m,r-1)}(k)  \chi_{C_m}(k)}{k(k-1)}.
\end{equation}
Using Equations~(\ref{eq4}) and ~(\ref{eq5}) in Equation~(\ref{eq6}), we get 
\begin{equation*}
\chi_{B(m,r)}(k) = \frac{ [(k-1)^{m} + (-1)^{m}(k-1)]^{r}}{ [k(k-1)]^{r-1} }.
\end{equation*}
Hence the proof follows by induction.
\end{proof}

Recall that $B(m,n)$ is a signed graph with empty signature, and so the ordinary chromatic polynomial and the signed chromatic polynomial of this graph are same. Note that $B_{1}(m,1) \cong C_{m}^{-} \cong B_{1}^{*}(m,1)$, and the chromatic polynomial $\chi_{C_{m}^{-}}(2k+1)$ is already obtained in the Lemma~\ref{C-}. Now for $n \geq 2$, we recursively determine the chromatic polynomials of the signed Book graphs $B_{1}(m,n)$ and $B_{1}^{*}(m,1)$ in the following two theorems.

\begin{theorem} \label{B_1}
For $n \geq 2$, the chromatic polynomial of $B_{1}(m,n)$ is given by
\begin{equation*}
 \chi_{B_{1}(m,n)}(2k+1) = \Big[ \sum_{i=0}^{m-2} (-1)^{i} (2k)^{(m-2)-i} \Big] \chi_{B_{1}(m,n-1)}(2k+1).
\end{equation*} 
\end{theorem}
\begin{proof}
Let $B_{1}(m,n)$ be the signed Book graph with signature $\{uu_{1}^{1}\}$, and consider $e_{1} = uu_{1}^{n}$. Using edge deletion-contraction formula on $e_{1}$, we get $$\chi_{B_{1}(m,n)}(2k+1) = \chi_{B_{1}'(m,n-1)}(2k+1) - \chi_{B_{1}''(m,n)}(2k+1),$$ where $B_{1}'(m,n-1)$ denotes a graph obtained from $B_{1}(m,n-1)$ by attaching a path $P_{m-1}$ at vertex $v$ and $B_{1}''(m,n)$ denotes a graph which is almost same as $B_{1}(m,n)$ but one of its cycle is of length $m-1$. 

Using Lemma~\ref{Path_add}, the chromatic polynomial of the signed graph $B_{1}'(m,n-1)$ can be calculated in terms of $\chi_{B_{1}(m,n-1)}(2k+1)$. For $B_{1}''(m,n)$, we apply the edge deletion-contraction formula again on the edge $e_{2} = uu_{2}^{n}$. Thus repeated application of edge deletion-contraction formula and Lemma~\ref{Path_add} imply that the chromatic polynomial of $B_{1}(m,n)$ is
\begin{align*}
\chi_{B_{1}(m,n)}(2k+1) & = (2k)^{m-2} \chi_{B_{1}(m,n-1)}(2k+1) - (2k)^{m-3} \chi_{B_{1}(m,n-1)}(2k+1)+ \cdots + \\
&~~~~  (-1)^{m-2} (2k)^{(m-2)-(m-2)} \chi_{B_{1}(m,n-1)}(2k+1) \\
 & = \Big[ \sum_{i=0}^{m-2} (-1)^{i} (2k)^{(m-2)-i} \Big] \chi_{B_{1}(m,n-1)}(2k+1).
\end{align*}
This completes the proof.
\end{proof}

\begin{theorem} \label{B_{1}*}
For $n \geq 2$, the chromatic polynomial of $B_{1}^{*}(m,n)$, is given by
\begin{equation*}
 \chi_{B^{*}_{1}(m,n)}(2k+1) = \Big[ \sum_{i=0}^{m-3} (-1)^{i} (2k)^{(m-2)-i} \Big] \chi_{B^{*}_{1}(m,n-1)}(2k+1) + (-1)^{m-2} \chi_{B_{m}^{n-1}}(2k+1).
\end{equation*} 
\end{theorem}
\begin{proof}
Let $B^{*}_{1}(m,n)$ be the signed Book graph with signature $\{uv\}$, and consider $e_{1} = uu_{1}^{n}$. Using edge deletion-contraction formula on $e_{1}$, we have $$\chi_{B^{*}_{1}(m,n)}(2k+1) = \chi_{\widetilde{B}_1^*(m,n-1)}(2k+1) - \chi_{\widetilde{\widetilde{B}}_1^*(m,n)}(2k+1),$$
 where $\widetilde{B}_1^*(m,n-1)$ denotes a graph which is obtained from $B^{*}_{1}(m,n-1)$ by attaching a path $P_{m-1}$ at the vertex $v$ and the graph $\widetilde{\widetilde{B}}_1^*(m,n)$ denotes a graph which is almost same as $B^{*}_{1}(m,n)$ but one of its cycles is of length $m-1$.
  
Using Lemma~\ref{Path_add}, the chromatic polynomial of the signed graph $\widetilde{B}_1^*(m,n-1)$ can be obtained in terms of $\chi_{B^{*}_{1}(m,n-1)}(2k+1)$. For the graph $\widetilde{\widetilde{B}}_1^*(m,n)$, we apply the edge deletion-contraction formula again on the edge $e_{2} = uu_{2}^{n}$. 

Thus, by the repeated application of edge deletion-contraction formula and Lemma~\ref{Path_add}, we see that the chromatic polynomial of $B^{*}_{1}(m,n)$ is 
\begin{align*}
\chi_{B^{*}_{1}(m,n)}(2k+1) & = (2k)^{m-2} \chi_{B^{*}_{1}(m,n-1)}(2k+1) - (2k)^{m-3} \chi_{B^{*}_{1}(m,n-1)}(2k+1)+ \cdots + \\
&~~~ (-1)^{m-3} (2k)^{(m-2)-(m-3)} \chi_{B^{*}_{1}(m,n-1)}(2k+1) + (-1)^{m-2} (2k)^{(m-2)-(m-2)} \chi_{B_{m}^{n-1}}(2k+1)\\
  & = \Big[ \sum_{i=0}^{m-3} (-1)^{i} (2k)^{(m-2)-i} \Big] \chi_{B^{*}_{1}(m,n-1)}(2k+1) + (-1)^{m-2} \chi_{B_{m}^{n-1}}(2k+1).
\end{align*}
Note that, in the last step of edge deletion-contraction formula, the resultant graph is nothing but the graph $B_{m}^{n-1}$. This completes the proof.
\end{proof}

\noindent
Now we give a recursive formula for the chromatic polynomials of $B_{l}(m,n)$ and $B_{l}^{*}(m,n)$ in the following two theorems.

\begin{theorem} \label{B_l}
For $n \geq 2$ and $2 \leq l \leq \lceil \frac{n}{2} \rceil$, the chromatic polynomial of $B_{l}(m,n)$ is given by  
\begin{equation*}
 \chi_{B_{l}(m,n)}(2k+1) = \Big[ \sum_{i=0}^{m-2} (-1)^{i} (2k)^{(m-2)-i} \Big] \chi_{B_{l}(m,n-1)}(2k+1).
\end{equation*} 
\end{theorem}
\begin{proof}
The proof requires the same steps as we performed in the proof of Theorem~\ref{B_1}. 
\end{proof}

We note that the expression for the chromatic polynomial of $B_{l}(m,n)$, where $l$ satisfies $\lceil \frac{n}{2} \rceil < l \leq n$, is same as the expression given in Theorem~\ref{B_l}. So we can use Theorem~\ref{B_l} for all $l$ satisfying $1 \leq l \leq n$.

\begin{theorem} \label{B_l*}
For $n \geq 2$ and $2 \leq l \leq \lceil \frac{n}{2} \rceil$, the chromatic polynomial of $B^{*}_{l}(m,n)$ is given by  
\begin{equation*}
 \chi_{B^{*}_{l}(m,n)}(2k+1) = \Big[ \sum_{i=0}^{m-2} (-1)^{i} (2k)^{(m-2)-i} \Big] \chi_{B^{*}_{l-1}(m,n-1)}(2k+1).
\end{equation*} 
\end{theorem}
\begin{proof}
By switching at $u$, we see that the graph $B^{*}_{l}(m,n)$ is switching equivalent to the graph $B_{l^{\prime}}(m,n)$, where $l^{\prime} = (n+1)-l$. Since the chromatic polynomial of a signed graph is switching invariant, we have $\chi_{B^{*}_{l}(m,n)}(2k+1) = \chi_{B_{l^{\prime}}(m,n)}(2k+1)$. By Theorem~\ref{B_l}, we get 
\begin{equation} \label{eq7}
\chi_{B_{l^{\prime}}(m,n)}(2k+1) = \Big[ \sum_{i=0}^{m-2} (-1)^{i} (2k)^{(m-2)-i} \Big] \chi_{B_{l^{\prime}}(m,n-1)}(2k+1).
\end{equation} 
By switching at $u$, we see that the graph $B_{l^{\prime}}(m,n-1)$ is switching equivalent to $B^{*}_{n-l^{\prime}}(m,n-1)$. Since $n - l^{\prime} = l-1$, we have $$\chi_{B_{l^{\prime}}(m,n-1)}(2k+1) = \chi_{B^{*}_{n-l^{\prime}}(m,n-1)}(2k+1) = \chi_{B^{*}_{l-1}(m,n-1)}(2k+1).$$ Hence from Equation~(\ref{eq7}), we get 
\begin{equation*}
\chi_{B^{*}_{l}(m,n)}(2k+1) = \Big[ \sum_{i=0}^{m-2} (-1)^{i} (2k)^{(m-2)-i} \Big] \chi_{B^{*}_{l-1}(m,n-1)}(2k+1).
\end{equation*} 
This completes the proof.
\end{proof}

\section{Balanced Chromatic Polynomial of Signed Book Graph}\label{balanced-poly}

Recall that the \textit{balanced (zero-free)} chromatic polynomial $\chi_{(G,\sigma)}^{b}(\lambda)$ of a signed graph $(G,\sigma)$, defined for even positive arguments $\lambda = 2k$, is the function whose value equals the number of proper zero-free signed colorings of $(G,\sigma)$ in $2k$ signed colors.

In \cite{Zaslavsky2}, the author explained that the ordinary chromatic polynomial and the balanced chromatic polynomial of a signed graph $(G,\sigma)$ are different unless $(G,\sigma)$ is balanced. Let $(G,\sigma)$ be a signed graph with vertex set $\{1,2, \cdots, n\}$ and $B_{(G,\sigma)}$ be its signed graphic arrangement. We denote the collection $B_{(G,\sigma)} \cup \{h_{1}, h_{2}, \cdots, h_{n}\}$ by $B_{(G,\sigma)}^{0}$. Consequently, a signed coloring $c$ of a signed graph is proper and balanced (zero-free) if and only if as a point in $\mathbb{Z}^{n}$, it avoids each hyperplane of $B_{(G,\sigma)}^{0}$. 

For example, Let $C_{3}^{-}$ be an unbalanced cycle of length three as shown in Figure~\ref{C3}, and its signed graphic hyperplane arrangement is $B_{C_{3}^{-}} = \{h_{12}^{+}, h_{13}^{+}, h_{23}^{-}\}$. Thus $B_{C_{3}^{-}}^{0} = \{h_{12}^{+}, h_{13}^{+}, h_{23}^{-}, h_{1}, h_{2}, h_{3}\}$. The poset $\mathcal{P}(C_{3}^{-}) = \{ h_{12}^{-}, h_{13}^{-}, h_{23}^{+} \}$ is shown as a Hasse diagram in Figure~\ref{lat3}

\begin{figure}[h]
\centering 
\begin{tikzpicture}[scale=0.6]
\node[vertex] (v1) at (10,0) {};
\node[vertex] (v2) at (14,0) {};
\node[vertex] (v3) at (12,2.5) {};
\node at (10,-0.6) {2};
\node at (14,-0.6) {1};
\node at (12,3.1) {3};
\foreach \from/\to in {v1/v2,v2/v3} \draw (\from) -- (\to);
\draw [dashed] (10,0) -- (12,2.5);
\end{tikzpicture}
\caption{An unbalanced cycle of length three} \label{C3}
\end{figure}
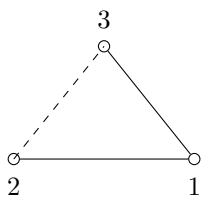

Theorem~\ref{Brian2} is also true in case of balanced chromatic polynomial, and we state this as follows.

\begin{theorem}\cite{Brian}\label{Brian3}
The balanced chromatic polynomial of a signed graph $(G,\sigma)$ on $n$ vertices is given by 
\[\chi_{(G,\sigma)}^b(2k) = \sum_{i=0}^{n} w_{i}^{\mathcal{P}(G,\sigma)} 2^{n-i} (k)_{n-i}.\]
\end{theorem}

\begin{figure}[h]
\centering
\begin{tikzpicture}
\node[vertex] (v1) at (0,0) {};
\node[vertex] (v1) at (-1.8,1.3) {};
\node[vertex] (v1) at (0,1.3) {};
\node[vertex] (v1) at (1.8,1.3) {};
\node[vertex] (v1) at (0,2.67) {};
\draw [thick] (0.05,0.03) -- (1.75,1.26);
\draw [thick] (-0.05,0.03) -- (-1.75,1.26);
\draw [thick] (0,0.05) -- (0,1.24);
\draw [thick] (0,1.36) -- (0,2.6);
\draw [thick] (1.75,1.34) -- (0.05,2.64);
\draw [thick] (-1.75,1.34) -- (-0.05,2.64);
\node at (0,-0.25) {\small{$\mathbb{R}^3$}};
\node at (2.4,1.3)  {\tiny{$x_{2}=+x_{3}$}};
\node at (-2.4,1.3) {\tiny{$x_{1}=-x_{2}$}}; 
\node at (0.55,1.3) {\tiny{$x_{1}=-x_{3}$}}; 
\node at (0,2.9) {\tiny{$x_1=-x_2=-x_3$}};
\end{tikzpicture}
\caption{The poset $\mathcal{P}(C_{3}^{-})$} \label{lat3}
\end{figure}
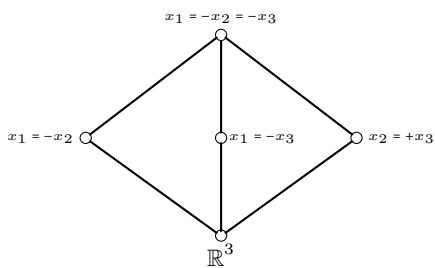 

It is clear that $w_{0}^{\mathcal{P}(C_{3}^{-})} = 1, w_{1}^{\mathcal{P}(C_{3}^{-})} = 3, w_{2}^{\mathcal{P}(C_{3}^{-})} = 1$, and $w_{3}^{\mathcal{P}(C_{3}^{-})} = 0$. 
Thus by Theorem~\ref{Brian3}, the balanced chromatic polynomial of $C_{3}^{-}$ is given by 
\begin{align*}
\chi_{C_{3}^{-}}^{b}(2k) & = 2^3(k)_3 + 3 \cdot 2^2(k)_2 + 2^1(k)_1 \\
 & = 8k(k-1)(k-2) + 12k(k-1) + 2k \\
 & = 8k^3-12k^2+6k.
\end{align*}

Similarly, the balanced chromatic polynomial of an unbalanced cycle of length two is obtained as
\begin{equation*}
\chi_{C_{2}^{-}}^{b}(2k) = 4k^2-4k.
\end{equation*}

Using Theorem~\ref{E-C}, we now give the recursive formula for balanced chromatic polynomial of unbalanced cycle of length $n$. 

\begin{lemm} \label{C_n}
For each $n \geq 2$, the balanced chromatic polynomial of an unbalanced cycle $C_{n}^{-}$ is given by 
$$\chi_{C_{n}^{-}}^{b}(2k) = 2k(2k-1)^{n-1} - \chi_{C_{n-1}^{-}}^{b}(2k).$$ 
\end{lemm}

We reformulate Lemma~\ref{Path_add} in terms of the balanced chromatic polynomial.

\begin{lemm} \label{Path_add1}
Let $(G,\sigma)$ be a signed graph. Then $$\chi_{(G,\sigma)^{l+1}}^{b}(2k) = (2k-1)^{l} \chi_{(G,\sigma)}^{b}(2k).$$
\end{lemm} 

By repeated use of Theorem~\ref{E-C} and Lemma~\ref{Path_add1}, we rewrite Lemma~\ref{B_m} in terms of balanced chromatic polynomial.

\begin{lemm} \label{B_m1}
For $m \geq 2$, the balanced chromatic polynomial of $B_{m}^{1}$ is given by
 $$\chi_{B_{m}^{1}}^{b}(2k) = \Big[ \sum_{i=0}^{m-2} (-1)^{i} (2k-1)^{(m-2)-i} \Big]  \chi_{C_{2}^{-}}^{b}(2k)$$
\end{lemm}

Now we write all the results of Section~\ref{chrom-poly} in terms of the balanced chromatic polynomials and omit the proofs as all these proofs are analogous to the corresponding proofs in Section~\ref{chrom-poly}.

\begin{prop}\label{prop_1-1}
For $n \geq 2$, the balanced chromatic polynomial of $B_{m}^{n}$ is given by 
\[ \chi_{B_{m}^{n}}^{b}(2k) = \Big[ \sum_{i=0}^{m-2} (-1)^{i} (2k-1)^{(m-2)-i} \Big] \chi_{B_{m}^{n-1}}^{b}(2k).\]
\end{prop}

\begin{theorem} \label{B_1-1}
For $n \geq 2$, the balanced chromatic polynomial of $B_{1}(m,n)$ is given by
\[ \chi_{B_{1}(m,n)}^{b}(2k) = \Big[ \sum_{i=0}^{m-2} (-1)^{i} (2k-1)^{(m-2)-i} \Big] \chi_{B_{1}(m,n-1)}^{b}(2k).\]
\end{theorem}

\begin{theorem} \label{B_{1}*-1}
For $n \geq 2$, the balanced chromatic polynomial of $B_{1}^{*}(m,n)$ is given by
\[ \chi_{B^{*}_{1}(m,n)}^{b}(2k) = \Big[ \sum_{i=0}^{m-3} (-1)^{i} (2k-1)^{(m-2)-i} \Big] \chi_{B^{*}_{1}(m,n-1)}^{b}(2k) + (-1)^{m-2} \chi_{B_{m}^{n-1}}^{b}(2k).\]
\end{theorem}

\begin{theorem} \label{B_l-1}
For $n \geq 2$ and $2 \leq l \leq \lceil \frac{n}{2} \rceil$, the balanced chromatic polynomial of $B_{l}(m,n)$ is given by  
\[ \chi_{B_{l}(m,n)}^{b}(2k) = \Big[ \sum_{i=0}^{m-2} (-1)^{i} (2k-1)^{(m-2)-i} \Big] \chi_{B_{l}(m,n-1)}^{b}(2k).\]
\end{theorem}

\begin{theorem} \label{B_l*-1}
For $n \geq 2$ and $2 \leq l \leq \lceil \frac{n}{2} \rceil$, the balanced chromatic polynomial of $B^{*}_{l}(m,n)$ is given by  
\[ \chi_{B^{*}_{l}(m,n)}^{b}(2k) = \Big[ \sum_{i=0}^{m-2} (-1)^{i} (2k-1)^{(m-2)-i} \Big] \chi_{B^{*}_{l-1}(m,n-1)}^{b}(2k).\]
\end{theorem}

\section{Conclusion}

We mention some basic applications of signatures on Book graph. Note that $B(4,n)$ resembles a book as we read in our daily life with $n$ number of pages. We all have our own ways of keeping our mind in each page of a book before closing it, but often that method is not reliable. However, an easy solution is to put a bookmark where we left off. So a signature of size one of the form $uu_{i}^{1}$, where $1 \leq i \leq n$, can be used as a bookmark to make sure that we pick where we left off. Similarly, a signature of size two of the form $\{uu_{i}^{1}, uu_{j}^{1} \}$, where $1 \leq i < j \leq n $, can be used for two bookmarks where one denotes the starting point of reading and second denote the place where we left off.

\end{document}